\documentclass{article}
\usepackage{graphicx} 

\usepackage{fullpage}

\usepackage{titling}

\usepackage{mathUtil}

\usepackage{fontawesome5}
\usepackage{xcolor}
\usepackage[outline]{contour}

\usepackage[backend=biber,style=numeric,maxnames=50]{biblatex}
\usepackage{hyperref}
\hypersetup{
    colorlinks=true,
    linkcolor=blue,
    filecolor=magenta,
    urlcolor=cyan,
    pdftitle={Structural Completeness in bi-IPC},
    pdfpagemode=FullScreen,
}
\title{Structural Completeness in bi-$\mathsf{IPC}$}

\author{Rodrigo Nicolau Almeida and Nick Bezhanishvili}

\date{\today}

\begin{document}

\maketitle

The purpose of this note\footnote{This work was originally developed in 2023, and was as of yet available only at the first author's website.  We have uploaded this note to make the results more easily accessible.} is to prove some basic results concerning structural completeness in bi-intuitionistic logic, and identify the source of such negative results. For a general introduction to the concepts of structural completeness in varieties, we point the reader to \cite{Rybakov1997-or}. We assume throughout familiarity with algebraic arguments, as well as with Esakia duality for bi-Heyting algebras (see \cite{Esakiach2019HeyAlg}).

Most of our notation is standard: we use $\mathbb{Q},\mathbb{V},\mathbb{S},\mathbb{P},\mathbb{P}_{U}$ to denote the closure operators corresponding to quasivarieties, varieties, closure under subalgebras, products and ultraproducts, respectively. We reserve boldface symbols (e.g. $\mathbf{K},\mathbf{H}$) for classes of algebras, and denote by $\alg{A},\alg{B}$ algebras. We write $\mathbf{n}$ for the bi-Heyting algebra with $n$ elements. We use symbols $L,L'$ for logics, especially those in $\mathsf{Ext}(\mathsf{biIPC})$, i.e., bi-intuitionistic logic. We write $\mathsf{CPC}$ for classical logic, understood as a bi-intuitionistic logic. Given a class of (bi-Heyting) algebras $\mathbf{K}$ we write $\mathsf{Log}(\mathbf{K})$ for the logic of $\mathbf{K}$, the set of formulas valid on all these algebras. We denote by $\Lambda(\mathbf{biHA})$ the lattice of extensions of bi-Heyting algebras.

The first fact we will need results from work of Davey, Kowalwski and Taylor \cite{Davey2021}:

\begin{fact}
The variety $\mathbb{V}(\bf{3})$ is the unique cover of the variety of Boolean algebras in $\Lambda(\mathbf{biHA})$.
\end{fact}

We also recall the following theorem (see e.g. \cite[Theorem 2.6]{bergmanstructuralcompleteness} for a more general statement):

\begin{theorem}
    Let $\bf{K}$ be a variety of bi-Heyting algebras. Then the following are equivalent:
    \begin{enumerate}
        \item $\mathsf{Log}(\bf{K})$ is structurally complete;
        \item Whenever $\bf{G}\subseteq \bf{K}$ is a proper quasivariety, then $\mathbb{V}(\bf{G})$ is a proper subvariety of $\mathbb{V}(\bf{K})$.
    \end{enumerate}
\end{theorem}

\begin{theorem}
For each $L\in \mathsf{Ext}(\mathsf{biIPC})$, if $L$ is structurally complete, then $L=\mathsf{CPC}$.
\end{theorem}
\begin{proof}
    Assume that $\bf{K}$ is any variety which is not Boolean. Hence by the above proposition, $\bf{3}\in \bf{K}$. Let $\mathbf{K}_{SI}$ be the subdirectly irreducible elements of $\mathbf{K}$. Now consider
    \begin{equation*}
        \mathbb{Q}(\{\alg{B} : \alg{B}\cong \alg{H}\times \mathbf{2}, \alg{H}\in \mathbf{K}_{SI}\}).
    \end{equation*}
Clearly the variety generated by this quasivariety is not proper. But the quasivariety itself is proper, since we claim that $\mathbf{3}$ is not there. To see this, note that if it were, since it is finite and subdirectly irreducible, it would belong to
\begin{equation*}
    \mathbb{SP}_{U}(\{\alg{B} : \alg{B}\cong \alg{H}\times \mathbf{2}, \alg{H}\in \mathbf{K}_{SI}\})
\end{equation*}
Now note that $\mathbf{3}$ will embed into the ultraproduct of these algebras, which means that it embeds into one of them (by Los' theorem). But this in turn means that $\mathbf{3}\leq \alg{B}$ for one of these algebras. Thinking dually we can see that this is a contradiction: if $X$ is any such Esakia space, it will be isomorphic to $Y\sqcup \{\bullet\}$; but then the loose point can be mapped to nowhere in $\mathbf{3}$. This shows the quasivariety is proper, and so the logic is not structurally complete.   
\end{proof}

An example of one rule which is admissible but not derivable in every variety is the following
\begin{equation*}
    \neg x=0 \ \wedge \  \sim x=1 \vdash 0=1
\end{equation*}
where ``$=$" abbreviates the if and only if. To see why, note that this holds in every algebra of the form $\mathbf{B}\times \mathbf{2}$, but as a rule it cannot be valid in any variety, since $\mathbf{3}$ will be in all of these varieties, and it does not derive this rule. We note that incidentally this also shows a stronger result.

\begin{theorem}
    Let $L\in \mathsf{Ext}(\mathsf{biIPC}$). Then $L$ is passively structurally complete if whenever $\mathbf{A},\mathbf{B}\in \mathsf{Var}(L)$, then $Th_{P}(\mathbf{A})=Th_{P}(\mathbf{B})$ where $Th_{P}(\mathbf{A})$ is the positive existential theory.
\end{theorem}
\begin{proof}
    See \cite[Theorem 7.2]{Moraschini2020}.
\end{proof}

\begin{corollary}
There is no passively structurally complete bi-IPC logic except for classical logic.
\end{corollary}
\begin{proof}
    Note that $\mathbf{3}$ and $\mathbf{3}\times \mathbf{2}$ will not have the same positive existential theory, since one satisfies the antecedent of the above rule for some $x$, and the other does not.
\end{proof}

It seems that passivity is in fact the only thing stopping the variety generated by $\mathbf{3}$ from being structurally complete. For this, we first note the fact (which will follow from the analysis in the next chapter) that the algebra $\mathbf{3}\times \mathbf{2}\times \mathbf{2}$ is in fact isomorphic to the free algebra on the variety $\mathsf{Var}(\mathbf{3})$. We also make use of the following criterion by Stronkowski:

\begin{proposition}
    Let $\alg{C}$ be a subalgebra of the free algebra $\bf{F}_{\omega}$ of a variety $\bf{K}$. Then $\bf{K}$ is actively structurally complete if and only if for each $\alg{D}$, a subdirectly irreducible algebra, $\alg{D}\times \alg{C}\in \mathbb{Q}(\bf{F}_{\omega})$.
\end{proposition}

\begin{corollary}
    The variety $\mathsf{Var}(\mathbf{3})$ is actively structurally complete.
\end{corollary}
\begin{proof}
By Jonnson's lemma, we know that the subdirect irreducibles will be $\mathbf{3}$ and $\mathbf{2}$. Then we must show that:
\begin{enumerate}
    \item $\mathbf{3}\times \mathbf{F}(1)\in \mathbb{Q}(\bf{F}_{\omega})$;
    \item $\mathbf{2}\times \mathbf{F}(1)\in \mathbb{Q}(\bf{F}_{\omega})$.
\end{enumerate}
By duality, one can check that $\bf{F}(1)$ is a two element chain, with two spare points (by simply coloring the dual space and noting that this terminates as soon as this poset is obtained); taking the product with itself we obtain two, two element chains and four spare points, and certainly we can either forget one of the one lines and a point, or forget all but two of the points. This obtains that these algebras are in $\mathbb{Q}(\bf{F}(1))\subseteq \mathbb{Q}(\bf{F}_{\omega})$ (noting that $\mathbf{F}(1)$ embeds in $\mathbf{F}(\omega)$), which was to show.
\end{proof}

Using very similar arguments one can likewise show that $\mathsf{Var}(\bf{n})$ is actively structurally complete for each natural number $n$. In fact we can show:

\begin{theorem}
    The variety $\mathsf{Lin}$ of all linear bi-Heyting algebras is hereditarily actively structurally complete.
\end{theorem}
\begin{proof}
    Note that its subvarieties will be precisely $\mathsf{Var}(\bf{n})$ for each natural number $n$, which as mentioned are actively structurally complete; hence it suffices to show that $\mathsf{Lin}$ will be actively structurally complete. This can be achieved by noting that the subdirect irreducibles of this logic will be precisely the chains, to which the previous argument easily applies.
\end{proof}

\begin{question}
    Can one provide a (Citkin-style) characterization of hereditarily active structurally complete extensions of $\mathsf{biIPC}$?
\end{question}

\printbibliography

@book{Rybakov1997-or,
  title = "Admissibility of Logical Inference Rules",
  author    = "Rybakov, Vladimir V",
  publisher = "Elsevier",
  year      =  1997,
  copyright = "http://creativecommons.org/licenses/by-nc-nd/4.0/"
}

@inbook{bergmanstructuralcompleteness,
author={Clifford Bergman},
title={Structural Completeness in Algebra and Logic},
booktitle={Algebraic logic},
publisher = {North Holland},
editor = {H. Andr\'{e}ka and James Donald Monk and I. N\'{e}meti},
	title = {Algebraic Logic},
	year = {1991},
pages={59-73}
}

@article{Moraschini2020,
  title = {Singly generated quasivarieties and residuated structures},
  volume = {66},
  ISSN = {1521-3870},
  url = {http://dx.doi.org/10.1002/malq.201900012},
  DOI = {10.1002/malq.201900012},
  number = {2},
  journal = {Mathematical Logic Quarterly},
  publisher = {Wiley},
  author = {Moraschini,  Tommaso and Raftery,  James G. and Wannenburg,  Johann J.},
  year = {2020},
  month = jun,
  pages = {150–172}
}

@article{Davey2021,
  doi = {10.1142/s021819672150034x},
  url = {https://doi.org/10.1142/s021819672150034x},
  year = {2021},
  month = jun,
  publisher = {World Scientific Pub Co Pte Lt},
  volume = {31},
  number = {04},
  pages = {727--774},
  author = {Brian A. Davey and Tomasz Kowalski and Christopher J. Taylor},
  title = {Splittings in varieties of logic},
  journal = {International Journal of Algebra and Computation}
}

@book{Esakiach2019HeyAlg,
	title = {Heyting Algebras: Duality Theory},
	publisher = {Springer},
	author = {Esakia,Leo},
	editor = {Bezhanishvili, Guram and Holliday, Wesley},
	note = {English translation of the original 1985 book},
	year = {2019}
}

\end{document}